\documentclass{amsart}
\usepackage{amsmath,amssymb,amsthm,bm}
\usepackage{graphicx,enumerate}
\usepackage{layout} 
\usepackage{longtable}

\theoremstyle{plain}
\newtheorem{theorem}{Theorem}[section]
\newtheorem{lemma}[theorem]{Lemma}
\newtheorem{proposition}[theorem]{Proposition}

\theoremstyle{definition}

\theoremstyle{remark}

\newcommand{\bZ}{\mathbb{Z}}
\newcommand{\bQ}{\mathbb{Q}}
\newcommand{\bR}{\mathbb{R}}
\newcommand{\bC}{\mathbb{C}}

\newcommand{\fO}{\mathfrak{O}}
\newcommand{\fA}{\mathfrak{A}}
\newcommand{\fH}{\mathfrak{H}} 

\title[graded ring of Siegel modular forms of degree two]
{On the graded ring of Siegel modular forms of degree two with respect to a non-split symplectic group}

\author{Hidetaka Kitayama}
\begin{document}

\begin{abstract} 
We will give the graded ring of Siegel modular forms of degree two 
with repsect to a certain discrete subgroup of a non-split symplectic group 
explicitly. 
\end{abstract}

\maketitle
\vspace*{-5mm}

%%%%%%%%%%%%%%%%%%%%%%%%%%%%%%%%%%%%%%%%%%%%%%%%%%%%%%%%%%%
%%%%%%%%%%%%%%%%%%%%%%%%%%%%%%%%%%%%%%%%%%%%%%%%%%%%%%%%%%%
%
\section{Introduction}\label{sec:intro}
%
%%%%%%%%%%%%%%%%%%%%%%%%%%%%%%%%%%%%%%%%%%%%%%%%%%%%%%%%%%%%
%%%%%%%%%%%%%%%%%%%%%%%%%%%%%%%%%%%%%%%%%%%%%%%%%%%%%%%%%%%%

The purpose of this paper is to give explicitly the graded ring of 
Siegel modular forms of degree two with respect to a certain discrete subgroup 
of a non-split symplectic group. 
(Theorem \ref{thm:main} below).  
In this section, we give an introduction for our main result and the way to prove it. 

Let $B$ be an indefinite quaternion algebra over $\bQ$ of discriminant $D$ 
with the canonical involution $\bar{\ }$. 
We define the group $U(2;B)$ as the unitary group with respect to the quaternion 
hermitian space of rank two, i.e. 
\[ U(2;B) := \left\{ g\in GL(2;B)\ \left|\ {}^t\overline{g} \begin{pmatrix} 0&1\\1&0 \end{pmatrix} g =\begin{pmatrix} 0&1\\1&0 \end{pmatrix} \right. \right\} , \] 
where ${}^t\overline{g}=\begin{pmatrix} \overline{a}&\overline{c} \\ \overline{b}&\overline{d} \end{pmatrix}$ 
for $g=\begin{pmatrix} a&b\\c&d \end{pmatrix}$. 
We can regard $U(2;B)$ as a subgroup of $Sp(2;\bR)$ by fixing an isomorphism 
$U(2;B)\otimes _{\bQ}\bR \simeq Sp(2;\bR)$. 
If $D\not= 1$, then $U(2;B)$ is a non-split $\bQ$-form of $Sp(2;\bR)$. 
Let $\mathfrak{O}$ be the maximal order of $B$, 
which is unique up to conjugation. 
If we take a positive divisor $D_1$ of $D$ and put $D_2:=D/D_1$, 
then there is the unique maximal two-sided ideal $\mathfrak{A}$ of $\mathfrak{O}$ 
such that $\mathfrak{A}\otimes _{\bZ}\bZ _p= \mathfrak{O}_p$ if $p\mid D_1$ or $p\nmid D$,  
and $\mathfrak{A}\otimes _{\bZ}\bZ _p= \pi\mathfrak{O}_p$ if $p\mid D_2$, 
where $\pi$ is a prime element of $\mathfrak{O}_p$.  
We treat a discrete subgroup of $Sp(2;\bR)$ defined by 
\[ \Gamma _{\fA}=\Gamma (D_1,D_2) := U(2;B) \cap 
\begin{pmatrix} \fO & \fA ^{-1} \\ \fA & \fO \end{pmatrix}. \] 

We are interested in studying Siegel modular forms with respect to $\Gamma_{\fA}$. 
We denote by $M_k(\Gamma)$ the space of Siegel modular forms 
of weight $k$ with respect to $\Gamma=\Gamma_{\fA}=\Gamma(D_1,D_2)$. 
The main theorem of this paper is the following: 
\begin{theorem} \label{thm:main} 
The graded ring of Siegel modular forms with resptect to $\Gamma(1,6)$ 
is given explicitly by 
\begin{align*} 
\displaystyle\bigoplus _{k=0}^{\infty} M_{k}(\Gamma(1,6)) 
 &= \bC [E_2,E_4,\chi_{5a},E_6] \oplus \chi_{5b}\bC [E_2,E_4,\chi_{5a},E_6] \\ 
 &{\hspace{5mm}} \oplus \chi_{15}\bC [E_2,E_4,\chi_{5a},E_6] \oplus \chi_{5b}\chi_{15}\bC [E_2,E_4,\chi_{5a},E_6],  
\end{align*}  
where we denote by $E_k$ ($k=2,4,6$) the Eisenstein series which are defined in 
\cite{Hir99}, 
and denote by $\chi_{5a}$, $\chi_{5b}$ and $\chi_{15}$ the Siegel cusp forms 
of weight $5$, $5$ and $15$ respectively, 
which are defined in Proposition \ref{prop:weight5,10} and \ref{prop:weight15,20} 
below. 
The four modular forms $E_2$, $E_4$, $\chi_{5a}$ and $E_6$ are
agebraically independent over $\bC$, 
and ${\chi_{5b}}^2$ and ${\chi_{15}}^2$ can be written by $E_2$, $E_4$, $\chi_{5a}$ 
and $E_6$. 
Fourier coefficients of these forms are computable and given in Appendix. 
\end{theorem} 

Explicit constructions of the graded ring of Siegel modular forms of split case  
have been studied by many authors, for example, 
Igusa\cite{Igu62}, Ibukiyama\cite{Ibu91}, Freitag and Salvati Manni\cite{FreSal04},  
Gunji\cite{Gun04} and Aoki and Ibukiyama\cite{AokIbu05}, 
but, as far as the author knows, 
no results were known for the case of non-split $\bQ$-forms of $Sp(2;\bR)$. 
We are short of available methods in the case of non-split $\bQ$-forms  
because they have only point cusps. 
Hirai \cite{Hir99} determined the spaces of low weights for $\Gamma(6,1)$ 
by using his explicit formula of Fourier coefficients of the Eisenstein series 
(cf. Proposition \ref{prop:HiraiEisen}), Oda lifting (cf. \cite{Oda77},\cite{Sug84}) 
and Hashimoto's explicit dimension formula (cf. \cite{Has84}),   
but he did not obtain the graded ring.  

Our motivations for this study are as follows. 
First, we are interested to study a possible correspondence between 
Siegel modular forms for split $\bQ$-forms and non-split $\bQ$-forms. 
In the case where $B$ is definite, 
Ibukiyama has been studying a generalization of Eichler-Jacquet-Langlands correspondence 
to the case of $Sp(2)$, 
by means of comparison of explicit dimension formulas in \cite{Ibu85},\cite{HI85},\cite{Ibu07}, and of Euler factors of spinor L-functions in \cite{Ibu84}. 
In our previous paper, 
we obtained an explicit dimension formula for our group $\Gamma_{\fA}$. 
In this paper, we construct Siegel modular forms for $\Gamma_{\fA}$ explicitly. 
We will give some examples of Euler factors for them on another occasion. 
Our results would be used for a similar comparison as \cite{Ibu84}. 
Second, we are interested in studying Siegel modular forms for non-split $\bQ$-forms 
explicitly. 
The result and method of this paper would be used in further investigations.  

We summarize the way to prove our main theorem, Theorem \ref{thm:main}. 
The dimension formula which we obtained in our previous work 
(see subsection \ref{subsec:dimension_formula}) 
plays a crucial role in our work. 
The first step to prove Theorem \ref{thm:main} is to determine 
the spaces of weight $k\leq 4$. 
Note that the formula can not be applied for the spaces of weight $k\leq 4$. 
We will prove Proposition \ref{prop:lowweight} in section \ref{sec:lowweight}. 
\begin{proposition} \label{prop:lowweight} 
\begin{align*} 
M_1(\Gamma (1,6)) &= \{ 0\} ,\hspace{10mm}   M_2(\Gamma (1,6)) = \bC E_2, \\ 
M_3(\Gamma (1,6)) &= \{ 0\} ,\hspace{10mm}  M_4(\Gamma (1,6)) = \bC {E_2}^2\oplus \bC E_4. 
\end{align*} 
\end{proposition} 

\noindent 
The second step to prove Theorem \ref{thm:main} is to construct 
$\chi_{5a}$, $\chi_{5b}$ and $\chi_{15}$. 
Generally speaking, it is difficult to construct modular forms of odd weight. 
As for $\chi_{5a}$ and $\chi_{5b}$, we will prove Proposition \ref{prop:weight5,10} 
in section \ref{sec:weight5,10}  
by detailed calculation of Fourier coefficients of the space of weight $10$ 

\begin{proposition} \label{prop:weight5,10} 
The Siegel cusp forms $\chi_{5a}$ and $\chi_{5b}$ of weight $5$ exist and 
are determined uniquely up to sign by the following relations:   
\begin{align*} 
{\chi_{5a}}^2 &= \tfrac{31513745731}{416023384089600}E_{10} -\tfrac{126433528597}{311423218947072}{E_2}^5+\tfrac{11304517601}{14285468759040}{E_2}^3E_4 \\ 
 &{\hspace{5mm}} -\tfrac{41742579637}{1557116094735360}{E_2}^2E_6-\tfrac{38947571}{120147846816}E_2{E_4}^2-\tfrac{1000259890201}{9083177219289600}E_4E_6, \\ 
{\chi_{5b}}^2 &= \tfrac{31513745731}{416023384089600}E_{10} +\tfrac{266799861}{1281577032704}{E_2}^5-\tfrac{261925781}{1587274306560}{E_2}^3E_4 \\ 
 &{\hspace{5mm}} -\tfrac{1914649869}{6407885163520}{E_2}^2E_6+\tfrac{935053847}{51903869824512}E_2{E_4}^2+\tfrac{551346719209}{3406191457233600}E_4E_6.  
\end{align*} 
\end{proposition} 
\noindent 
As for $\chi_{15}$, 
we will prove Proposition \ref{prop:weight15,20} in section \ref{sec:weight15,20}. 
We denote by $\{ E_2,E_4,\chi_{5a},E_6\}_*$ the Siegel cusp form of weight $20$ 
obtained from $E_2$, $E_4$, $\chi_{5a}$ and $E_6$ by the differential operator 
which is reviewed in subsection \ref{subsec:differential_operator}. 
\begin{proposition} \label{prop:weight15,20} 
The Siegel cusp form $\{ E_2,E_4,\chi_{5a},E_6\}_*$ is divisible by $\chi_{5b}$, 
so we can define $\chi_{15} := \{ E_2,E_4,\chi_{5a},E_6\}_* / \chi_{5b}$.  
\end{proposition} 
\noindent 
Fianlly, we will prove Theorem \ref{thm:main} in section \ref{sec:main}. 
We can obtain the generating function of $\mathrm{dim}_{\bC}M_k(\Gamma(1,6))$ 
by using the dimension formula and Proposition \ref{prop:lowweight}. 
It is crucial for the final step to prove the equality. 

The author would like to express his sincere gratitude to Professor Tomoyoshi Ibukiyama for giving him this problem and various advices. 
The author also would like to thank Professor Takashi Sugano and 
Professor Satoshi Wakatsuki for their many helpful advices. 
The author also would like to thank Professor Hiroki Aoki for giving him a various 
suggestion about the proof of Proposition \ref{prop:weight15,20}. 
The author is supported by the Grant-in-Aid for JSPS fellows.

%%%%%%%%%%%%%%%%%%%%%%%%%%%%%%%%%%%%%%%%%%%%%%%%%%%%%%%%%%%%
%%%%%%%%%%%%%%%%%%%%%%%%%%%%%%%%%%%%%%%%%%%%%%%%%%%%%%%%%%%%
%
\section{Preliminaries} \label{sec:preliminaries} 
%
%%%%%%%%%%%%%%%%%%%%%%%%%%%%%%%%%%%%%%%%%%%%%%%%%%%%%%%%%%%%
%%%%%%%%%%%%%%%%%%%%%%%%%%%%%%%%%%%%%%%%%%%%%%%%%%%%%%%%%%%% 

%%%%%%%%%%%%%%%%%%%%%%%%%%%%%%%%%%%%%%%%%%%%%%%%%%%%%%%%%%%%%%%%
%
\subsection{Siegel modular forms} \label{subsec:Siegel} 
%
%%%%%%%%%%%%%%%%%%%%%%%%%%%%%%%%%%%%%%%%%%%%%%%%%%%%%%%%%%%%%%%%

We review Sigel modular forms to fix notation. 
Let $Sp(2;\bR)$ be the real symplectic group of degree two, i.e. 
\[ Sp(2;\bR ) =\left\{ g\in GL(4,\bR )\ \left| \ g\begin{pmatrix} 0_2&1_2\\-1_2&0_2 \end{pmatrix} {}^tg=\begin{pmatrix} 0_2&1_2\\-1_2&0_2 \end{pmatrix} \right. \right\} . \] 
Let $\fH_2$ be the Siegel upper half space of degree two, i.e. 
\[ \fH _2=\{ Z\in M(2;\bC )\ |\ {}^tZ=Z,\ \mathrm{Im}(Z)\mbox{ is positive definite } \} . \] 
The group $Sp(2;\bR)$ acts on $\fH_2$ by 
\[ \gamma \langle Z\rangle :=(AZ+B)(CZ+D)^{-1} \] 
for any $\gamma =\begin{pmatrix} A&B\\C&D \end{pmatrix}\in Sp(2;\bR)$ 
and $Z\in \fH_2$.    
Let $\Gamma$ be a discrete subgroup of $Sp(2;\bR)$ such that vol($\Gamma\backslash \fH_2)<\infty$. 
We say that a holomorphic function $F(Z)$ on $\fH_2$ is a Sigel modular form 
of weight $k$ of $\Gamma$ if it satisfies 
\[ f(\gamma \langle Z\rangle )=\det (CZ+D)^k f(Z), \hspace{5mm} 
\text{for }\forall \gamma =\begin{pmatrix} A&B\\C&D \end{pmatrix} 
\in \Gamma, \forall Z\in \fH_2. \] 
If a Siegel modular form $F(Z)$ satisfies 
\[ \det (\mathrm{Im}(Z)^{1/2})|f(Z)| \text{ is bounded on }\fH_2,  \] 
then we say that $F(Z)$ is a Siegel cusp form. 
We denote by $M_k(\Gamma)$ (resp. $S_k(\Gamma)$) 
the spaces of all Siegel modular forms (resp. cusp forms) of weight $k$ of $\Gamma$. 
It is known that $M_k(\Gamma)$ and $S_k(\Gamma)$ are finite dimensional vector spaces over $\bC$.

%%%%%%%%%%%%%%%%%%%%%%%%%%%%%%%%%%%%%%%%%%%%%%%%%%%%%%%%%%%%%%%%
%
\subsection{Dimension formula} \label{subsec:dimension_formula} 
%
%%%%%%%%%%%%%%%%%%%%%%%%%%%%%%%%%%%%%%%%%%%%%%%%%%%%%%%%%%%%%%%%

Let $B$ be an indefinite quaternion algebra over $\bQ$. 
We fix an isomorphism $B\otimes _{\bQ}\bR\simeq M(2;\bR)$  
and we identify $B$ with a subalgebra of $M(2;\bR)$. 
We define $U(2;B)$ and $\Gamma(D_1,D_2)$ as in section \ref{sec:intro}. 
It is known that $U(2;B)\otimes _{\bQ}\bR$ is isomorphic to $Sp(2;\bR)$ by  
\[ \phi : U(2;B)\otimes _{\bQ}\bR \stackrel{\sim}{\longrightarrow}Sp(2;\bR) \] 
\[ \phi (g)=\begin{pmatrix} a_1&a_2&b_2&-b_1\\a_3&a_4&b_4&-b_3\\c_3&c_4&d_4&-d_3\\-c_1&-c_2&-d_2&d_1 \end{pmatrix} ,\quad   
g=\begin{pmatrix} A&B\\C&D \end{pmatrix} \in U(2;B)\otimes _{\bQ}\bR \] 
where $A=\begin{pmatrix} a_1&a_2\\a_3&a_4 \end{pmatrix}$, 
$B=\begin{pmatrix} b_1&b_2\\b_3&b_4 \end{pmatrix}$, 
$C=\begin{pmatrix} c_1&c_2\\c_3&c_4 \end{pmatrix}$, 
$D=\begin{pmatrix} d_1&d_2\\d_3&d_4 \end{pmatrix}\in$\nolinebreak $B$\nolinebreak $\otimes _{\bQ}$\nolinebreak $\bR$, 
and we can identify $\Gamma(D_1,D_2)$ with a discrete subgroup of $Sp(2;\bR)$ such that vol($\Gamma(D_1,D_2)\backslash \fH_2)<\infty$. 
 
In our previous paper \cite{Kit}, we obtained an explicit formula for dimensions of 
the spaces $S_k(\Gamma(D_1,D_2))$ of weight $k\geq 5$ for general $(D_1,D_2)$, 
including the vector-valued case.  
If we apply this formula to $S_k(\Gamma(1,2p))$ for an odd prime number $p$, 
then we have 
\begin{align*} 
\mathrm{dim}_{\bC}S_{k}(\Gamma (1,2p)) &= 
\frac{(k-2)(k-1)(2k-3)}{2^7\cdot 3^2\cdot 5}\cdot (p^2-1) +\frac{1}{2^3\cdot 3}\cdot (p-1)\\ 
 &{\quad } +\frac{(-1)^k(8+\big( \frac{-1}{p}\big) )+(2k-3)(8-\big( \tfrac{-1}{p}\big) )}{2^7\cdot 3}( p-\big( \tfrac{-1}{p}\big) ) \\ 
 &{\quad } +\frac{[0,-1,1;3]_k}{2^2\cdot 3^2} \cdot \Big( 4+\tfrac{1}{2}\big( \tfrac{-3}{p}\big) \big( 1-5\big( \tfrac{-3}{p}\big) \big)\Big) \big( p-\big( \tfrac{-3}{p}\big) \big) \\  
 &{\quad } +\frac{2k-3}{2^2\cdot 3^2} \cdot \Big( 5-\tfrac{1}{2}\big( \tfrac{-3}{p}\big) \big( 1+7\big( \tfrac{-3}{p}\big) \big)\Big) \big( p-\big( \tfrac{-3}{p}\big) \big) \\ 
 &{\quad } -\frac{1}{2^3}( 1-\big( \tfrac{-1}{p}\big) ) 
-\frac{1}{3}( 1-\big( \tfrac{-3}{p}\big) ) \\ 
 &{\quad } +\frac{2\cdot [1,0,0,-1,0;5]_k}{5}\cdot ( 1-\big( \tfrac{p}{5}\big) ) \\ 
 &{\quad } +\frac{[1,0,0,-1;4]_k}{2^2}\cdot \left\{ \begin{array}{ccc} 0&\cdots &\mbox{ if }p\equiv 1,7\mbox{ mod }8 \\ 1&\cdots &\mbox{ if }p\equiv 3,5\mbox{ mod }8 \end{array} \right. \\   
 &{\quad } +\frac{1}{6}\cdot 
\left\{ \begin{array}{ccl} 
        (-1)^k/2&\cdots &\mbox{ if }p=3 \\ 
        0&\cdots &\mbox{ if }p\equiv 1,11\mbox{ mod }12 \\ 
        \left[ 0,1,-1;3\right] _k &\cdots &\mbox{ if }p\equiv 5\mbox{ mod }12 \\ 
        (-1)^k&\cdots &\mbox{ if }p\equiv 7\mbox{ mod }12,  
\end{array} \right. 
\end{align*} 
where $\left( \frac{*}{*}\right)$ is the Legendre symbol and 
$[a_0,\ldots ,a_{m-1};m]_k$ is the function on $k$ 
which takes the value $a_i$ if $k\equiv i\mbox{ mod }m$. 
From this formula, we have $\mathrm{dim}_{\bC}S_k(\Gamma(1,6))$ as follows. 
Our formula is not valid for $k\leq 4$. 
In the following table, we formally substitute $k\leq 4$ in the formula.  

\begin{center} \scalebox{0.9}[1]{ 
\begin{tabular}{c||ccccc|ccccccccccccc} 
$k$&0&1&2&3&4&5&6&7&8&9&10&$\cdots$&15&$\cdots$&20&$\cdots$&25&$\cdots$ \\ \hline 
dim&0&$-1$&0&$-1$&1&2&2&2&3&4&6&$\cdots$&13&$\cdots$&27&$\cdots$&47&$\cdots$  
\end{tabular} } 
\end{center}

%%%%%%%%%%%%%%%%%%%%%%%%%%%%%%%%%%%%%%%%%%%%%%%%%%%%%%%%%%%%%%%%
%
\subsection{Fourier expansion} \label{subsec:Fourier_expansion} 
%
%%%%%%%%%%%%%%%%%%%%%%%%%%%%%%%%%%%%%%%%%%%%%%%%%%%%%%%%%%%%%%%%

Let $\fA$ be a maximal two-sided ideal of $\fO$. 
Since the class number of $\fO$ is one, we can write 
$\fA =\fO \pi =\pi \fO$ for some $\pi\in \fO$ such that $|N\pi |=D_1$ 
where $\fA$ corresponds to $(D_1,D_2)$ as in section \ref{sec:intro}.  
We define a three-dimensional $\bQ$ vector space 
$B^{0}:=\{ x\in B\ |\ \mathrm{Tr}(x)=0\}$ and 
define a lattice $\left(\fA ^{-1}\right)^0$ and its dual lattice by 
\begin{align*} 
\left(\fA ^{-1}\right) ^0 := B^{0}\cap \fA ^{-1}, \hspace{5pt} 
{\left(\fA ^{-1}\right) ^0} ^*:=\{ y\in B^0 \ |\ \mathrm{Tr}(xy)\in \bZ \mbox{ for any }x\in\left(\fA ^{-1}\right)^0 \}. 
\end{align*} 
Arakawa proved the following proposition in his master thesis 
\cite[Proposition 10]{Ara75} 
by the same way as, for example, Maa\ss \cite[\S 13]{Maa71}. 

\begin{proposition}[\cite{Ara75},\cite{Hir99}] \label{prop:Fourier} 
Let $\Gamma_{\fA}$ be the discrete subgroup of $Sp(2;\bR)$ defined in section \ref{sec:intro} and $k$ be a positive integer. 
Then $f(Z)\in M_k(\Gamma_{\fA})$ has the following Fourier expansion 
\[ f(Z)=C_f(0)+\sum_{\eta\in {\left( \fA ^{-1}\right)^0}^{*} \atop \eta J>0} C_f(\eta )\mathbf{e}[\mathrm{Tr}(\eta ZJ)], \hspace{10pt} (\mathbf{e}[z]:=e^{2\pi iz}) \] 
where {\small $J=\begin{pmatrix} 0&1\\-1&0 \end{pmatrix}$} and 
$\eta J>0$ means that $\eta J$ is positive definite when we regard $\eta$ as 
an element of $M(2;\bR)$.  
In particular, $f(Z)\in S_k(\Gamma_{\fA})$ is equivalent to $C_f(0)=0$. 
\end{proposition}

%%%%%%%%%%%%%%%%%%%%%%%%%%%%%%%%%%%%%%%%%%%%%%%%%%%%%%%%%%%%%%%%
%
\subsection{Eisenstein series} \label{subsec:Eisenstein_series} 
%
%%%%%%%%%%%%%%%%%%%%%%%%%%%%%%%%%%%%%%%%%%%%%%%%%%%%%%%%%%%%%%%%

By applying the method of Shimura \cite{Shi83}, 
Hirai \cite{Hir99} studied the Eisenstein series $E_k$ ($k\geq 2$: even) 
on $\Gamma(D_1,D_2)$ 
and obtained an explicit formula of Fourier coefficients of it. 
(Proposition \ref{prop:HiraiEisen} below). 
We define 
\[ {\left( \fA^{-1}\right)^0}^*_{\mathrm{prim}} := \big\{ \eta\in{\left( \fA^{-1}\right)^0}^* \ \big| \ n^{-1}\eta\not\in {\left( \fA^{-1}\right)^0}^* \text{ for any integer }n \big\}. \]  
For $\eta\in{\left( \fA^{-1}\right)^0}^*$, we denote by $d_{\eta}$ and $\chi_{\eta}$ 
the discriminant and the Dirichlet character of $\bQ(\eta)/\bQ$ 
and denote by $B_m$ (resp. $B_{m,\chi_{\eta}}$) the $m$-th Bernoulli (resp. the generalized Bernoulli) number. 
We define positive integers $a_{\eta}$ and $f_{\eta}$ by 
\[ a_{\eta}^{-1}\eta\in{\left( \fA^{-1}\right)^0}^*_{\mathrm{prim}},\quad 
(2a_{\eta}^{-1}\eta )^2=d_{\eta} {f_{\eta}}^2.  \] 
We put $a_{\eta,p}=\mathrm{ord}_p(a_{\eta})$, $f_{\eta,p}=\mathrm{ord}_p(f_{\eta})$. 
Then the following proposition holds. 
\begin{proposition}[\cite{Hir99} Theorem 3.10] \label{prop:HiraiEisen} 
Let $k$ be an even positive integer. Then the Eisenstein series $E_k$ has the 
following Fourier expansion. 
\[ E_k(Z)=1+\sum _{\eta\in {\left(\fA ^{-1}\right)^0}^* \atop \eta J>0} C(\eta ) \mathbf{e}[\mathrm{Tr}(\eta Z)], \] 
where \\ 
\resizebox{0.9\textwidth}{!}{
$\displaystyle C(\eta )=\frac{4kB_{k-1,\chi_{\eta}}}{B_kB_{2k-2}} 
\prod_{p|D_1}\frac{(1-\chi_{\eta}(p)p^{k-1})(1-\chi_{\eta}(p)p^{k-2})}{p^{2k-2}-1} 
\prod_{p|D_2}\frac{1}{p^{k-1}-1} \prod _p F_p(\eta ,k)$,} \\ 
\resizebox{\textwidth}{!}{
$F_p(\eta ,k)= \begin{cases} 
\displaystyle \sum _{t=0}^{a_{\eta,p}} p^{(2k-3)t}+(1+\chi_{\eta}(p))\sum _{t=0}^{a_{\eta,p}-1}p^{(2k-3)t+k-1} 
\hspace{10mm} \cdots \text{ if }p|D_1, \\ 
\displaystyle \sum _{t=0}^{a_{\eta,p}} p^{(2k-3)t}-\chi_{\eta}(p)\sum _{t=0}^{a_{\eta,p}-1}p^{(2k-3)t+k-2} 
\hspace{20mm} \cdots \text{ if }p|D_2, \\ 
\displaystyle \sum _{t=0}^{a_{\eta,p}} \left\{ \sum _{l=0}^{a_{\eta,p}+f_{\eta,p}-t}p^{(2k-3)l+(k-1)t}-\chi_{\eta}(p)
\sum _{l=0}^{a_{\eta,p}+f_{\eta,p}-t-1}p^{(2k-3)l+(k-1)t+k-2} \right\} \\ 
\hspace{80mm} \cdots \text{ if }p\not| D. 
\end{cases}$} 
\end{proposition}   

We see from Proposition \ref{prop:Fourier} that 
\begin{align*}  
M_k(\Gamma (D_1,D_2)) &= S_k(\Gamma (D_1,D_2)) & &\text{ if }k\text{ is odd, and} \\ 
M_k(\Gamma (D_1,D_2)) &= S_k(\Gamma (D_1,D_2)) \oplus \bC E_k & &\text{ if }k\text{ is even}.   
\end{align*}

%%%%%%%%%%%%%%%%%%%%%%%%%%%%%%%%%%%%%%%%%%%%%%%%%%%%%%%%%%%%%%%%
%
\subsection{Rankin-Cohen type differential operators} 
\label{subsec:differential_operator} 
%
%%%%%%%%%%%%%%%%%%%%%%%%%%%%%%%%%%%%%%%%%%%%%%%%%%%%%%%%%%%%%%%%

We quote the following proposition from Aoki and Ibukiyama\cite{AokIbu05}.\footnote{In \cite{AokIbu05}, Proposition \ref{prop:diffoperator} has been proved in the case of general degree $n$.} 
For $Z\in H_2$, we write the $(i,j)$ component of $Z$ by $z_{ij}$. 
For Siegel modular forms $f_i\in M_{k_i}(\Gamma)$ 
of weight $k_i$ ($1\leq i\leq 4$), 
we define a new function $\{ f_1,f_2,f_3,f_4\}_{*}$ by 
\[ \{ f_1,f_2,f_3,f_4\}_{*} 
= \begin{vmatrix} 
k_1f_1&k_2f_2&k_3f_3 &k_4f_4 \\[3pt]  
\tfrac{\partial f_1}{\partial z_{11}}&\tfrac{\partial f_2}{\partial z_{11}}&\tfrac{\partial f_3}{\partial z_{11}}&\tfrac{\partial f_4}{\partial z_{11}} \\[5pt] 
\tfrac{\partial f_1}{\partial z_{12}}&\tfrac{\partial f_2}{\partial z_{12}}&\tfrac{\partial f_3}{\partial z_{12}}&\tfrac{\partial f_4}{\partial z_{12}} \\[5pt] 
\tfrac{\partial f_1}{\partial z_{22}}&\tfrac{\partial f_2}{\partial z_{22}}&\tfrac{\partial f_3}{\partial z_{22}}&\tfrac{\partial f_4}{\partial z_{22}} 
\end{vmatrix}. \] 
\begin{proposition}[Aoki and Ibukiyama \cite{AokIbu05}]  
\label{prop:diffoperator} 
$(i)$\ The above function $\{ f_1,f_2,f_3,f_4\}_{*}$ is a 
Siegel cusp form of weight $k_1+k_2+k_3+k_4+3$. \\ 
$(ii)$\ $f_1,f_2,f_3,f_4$ are algebraiclly independent 
if and only if $\{ f_1,f_2,f_3,f_4\}_{*}\not= 0$. 
\end{proposition}

%%%%%%%%%%%%%%%%%%%%%%%%%%%%%%%%%%%%%%%%%%%%%%%%%%%%%%%%%%%%%%%%%
%%%%%%%%%%%%%%%%%%%%%%%%%%%%%%%%%%%%%%%%%%%%%%%%%%%%%%%%%%%%%%%%%
%
\section{Proof of Proposition \ref{prop:lowweight}} \label{sec:lowweight} 
%
%%%%%%%%%%%%%%%%%%%%%%%%%%%%%%%%%%%%%%%%%%%%%%%%%%%%%%%%%%%%%%%%%
%%%%%%%%%%%%%%%%%%%%%%%%%%%%%%%%%%%%%%%%%%%%%%%%%%%%%%%%%%%%%%%%%

In this section, we will prove Proposition \ref{prop:lowweight},  
that is, we will determine the spaces of weight $k\leq 4$. 
Note that the dimension formula is not valid for weight $k\leq 4$. 
(See subsection \ref{subsec:dimension_formula}). 

We prepare to calculate Fourier coefficients. 
If we put 
\[ B:=\bQ +\bQ a +\bQ b+\bQ ab, \hspace{10mm} a^2=6, b^2=5, ab=-ba, \] 
then $B$ is an indefinite quaternion algebra over $\bQ$ of discriminant $6$, 
which is unique up to isomorphism. 
Let $\fO$ be the maximal order of $B$, which is unique up to conjugacy. 
It is known by Ibukiyama \cite{Ibu72},\cite{Ibu82} that 
$\fO$ can be taken as 
\[ \fO =\bZ + \bZ \frac{1+b}{2}+\bZ \frac{a(1+b)}{2}+\bZ \frac{(1+a)b}{5}. \] 
If we put $\fA =a\fO$, then $\fA$ is the unique maximal two-sided ideal 
corresponding to $(1,6)$. 
By a straightforward calculation, we obtain 
\[ {\left(\fA ^{-1}\right)^0}^* =\bZ\frac{5a+b+ab}{10}+\bZ\frac{b}{2}+\bZ a. \] 
For $\eta =x(5a+b+ab)/60+yb/12+za/6\in{\left(\fA^{-1}\right)^0}^*$, 
we denote it by $\eta =[x,y,z]$ and 
we can see from a direct calculation that 
the condition $\eta J>0$ is equivalent to 
\begin{center} 
$\begin{cases} 
x>0, \hspace{3mm} \text{and} \\ 
m_{\eta}:=-(5x^2+5y^2+24z^2-2xy+24zx)>0.  
\end{cases}$ 
\end{center} 

We have the following modular forms which are obtained as 
products of Eisenstein series $E_k$'s: 
\begin{align*} 
\mathrm{weight}\ 2 &: E_2, & 
\mathrm{weight}\ 4 &: {E_2}^2, E_4, \\ 
\mathrm{weight}\ 6 &: {E_2}^3, E_2E_4, E_6, & 
\mathrm{weight}\ 8 &: {E_2}^4, {E_2}^2E_4, E_2E_6, {E_4}^2, E_8.  
\end{align*} 
For the sake of simplicity of Fourier coefficients, 
we use the following $\varphi _k$ instead of $E_k$ ($k=2,4,6,8$):  
\begin{align*} 
\varphi _2 &= E_2, & 
\varphi _4 &= -\tfrac{13}{288}\cdot (E_4-{\varphi_2}^2), \\ 
\varphi _6 &= -\tfrac{341}{113184}\cdot (E_6-{\varphi_2}^3)-\tfrac{109}{262}\cdot \varphi_2\varphi _4, & 
\varphi _8 &= 138811E_8, 
\end{align*} 
then we have Fourier coefficients of them as in the following tables.  

\vspace{2mm} 
\begin{center} 
\resizebox{0.8\textwidth}{!}{ 
\begin{tabular}{|c|c||c||c|c||c|c|c|} \hline \label{tab:weight2,4,6} 
$m_{\eta}$&$\eta$&$\varphi _2$&${\varphi _2}^2$&$\varphi _4$&${\varphi _2}^3$&$\varphi _2\varphi _4$&$\varphi _6$ \\ \hline\hline 
0&[0, 0, 0]&1&1&0&1&0&0 \\ \hline
3&[2, 1, $-$1]&48&96&1&144&1&0 \\ \hline
4&[2, 0, $-$1]&72&144&$-$1&216&$-$1&1 \\ \hline
12&[4, 2, $-$2]&192&2688&$-$2&7488&46&$-$6 \\ \hline
16&[4, 0, $-$2]&216&5616&6&16200&$-$66&42 \\ \hline
19&[4, 1, $-$2]&144&7200&$-$5&21168&19&$-$16 \\ \hline
24&[5, 1, $-$2]&288&15552&12&45792&$-$36&$-$60 \\ \hline
27&[6, 3, $-$3]&192&18816&36&166464&132&96 \\ \hline
36&[6, 0, $-$3]&360&41040&$-$45&495288&363&21 \\ \hline
40&[6, 2, $-$3]&288&52416&$-$4&654048&$-$580&100 \\ \hline
\end{tabular}  
} 
\end{center} 

\vspace{2mm} 
\begin{center} 
\resizebox{0.8\textwidth}{!}{ 
\begin{tabular}{|c|c||c|c|c|c|c|} \hline \label{tab:weight8} 
$m_{\eta}$&$\eta$&${\varphi _2}^4$&${\varphi _2}^2\varphi _4$&$\varphi _2\varphi _6$&${\varphi _4}^2$&$\varphi _8$ \\ \hline\hline 
0&[0, 0, 0]&1&0&0&0&138811 \\ \hline
3&[2, 1, $-$1]&192&1&0&0&13440 \\ \hline
4&[2, 0, $-$1]&288&$-$1&1&0&87840 \\ \hline
12&[4, 2, $-$2]&14592&94&$-$6&1&110974080 \\ \hline
16&[4, 0, $-$2]&31968&$-$138&114&1&719673120 \\ \hline
19&[4, 1, $-$2]&42048&43&32&$-$2&2181836160 \\ \hline
24&[5, 1, $-$2]&91008&$-$84&84&4&10043268480 \\ \hline
27&[6, 3, $-$3]&553728&2532&$-$192&$-$4&21427714560 \\ \hline
36&[6, 0, $-$3]&1736352&$-$4413&3261&$-$8&140109455520 \\ \hline
40&[6, 2, $-$3]&2302848&3452&$-$764&$-$4&277771616640 \\ \hline
43&[6, 1, $-$3]&3318336&$-$1613&992&18&441018218880 \\ \hline
48&[8, 4, $-$4]&11137152&16012&3396&82&909100536960 \\ \hline
51&[7, 2, $-$4]&7825536&3318&$-$4992&40&1337603408640 \\ \hline
52&[7, 1, $-$4]&9062784&$-$5116&$-$2180&$-$56&1528671231360 \\ \hline
64&[8, 0, $-$4]&49322592&12476&9556&90&5895562286880 \\ \hline
\end{tabular} 
} 
\end{center} 

\vspace{2mm} 
From these tables and the results of the dimension formula, 
we can see the following: 
\begin{align*} 
M_2(\Gamma (1,6)) &\supseteq \bC E_2, \hspace{10mm} 
M_4(\Gamma (1,6)) \supseteq \bC {E_2}^2\oplus \bC E_4, \\ 
M_6(\Gamma (1,6)) &= \bC {E_2}^3 \oplus \bC E_2E_4 \oplus \bC E_6,   \\ 
M_8(\Gamma (1,6)) &= \bC {E_2}^4 \oplus \bC {E_2}^2E_4 \oplus \bC E_2E_6 \oplus \bC {E_4}^2.  \\ 
 &\hspace{-10mm} \left( E_8=\tfrac{48860325}{18184241}{E_2}^4-\tfrac{107719950}{18184241}{E_2}^2E_4 
+\tfrac{26257000}{18184241}E_2E_6+\tfrac{387686}{138811}{E_4}^2 \right) 
\end{align*} 

We can prove Proposition \ref{prop:lowweight} 
by using the spaces of weight $6$ and $8$. 
We prove the following lemma.  
\begin{lemma} \label{lem:no_cusp_forms} 
If $k(\not= 6)$ is a positive divisor of $6$, 
then there are no non-zero cusp forms of weight $k$. 
\end{lemma} 

\begin{proof} 
We assume that there is a non-zero cusp form $f$ of weight $k$. 
Then the Fourier coefficients of $f^2\in S_{2k}(\Gamma (1,6))$ are: 
\begin{align*} 
C_{f^2}(0,0,0) &= C_{f}(0,0,0)\cdot C_f(0,0,0) =0, \\ 
C_{f^2}(2,1,-1) &= 2\cdot C_{f}(0,0,0)\cdot C_f(2,1,-1)=0, \\ 
C_{f^2}(2,0,-1) &= 2\cdot C_{f}(0,0,0)\cdot C_f(2,0,-1)=0,   
\end{align*} 
so the Fourier coefficients of $f^{6/k}\in S_6(\Gamma(1,6))$ are also 
\[ C_{f^{6/k}}(0,0,0)=C_{f^{6/k}}(2,1,-1)=C_{f^{6/k}}(2,0,-1)=0. \]  
Hence we have $f^{6/k}=0$ because of the table of Fourier coefficients of 
the space of weight $6$ on page \pageref{tab:weight2,4,6}, 
but this contradicts the assumption that $f$ is not zero.  
\end{proof} 

\noindent 
\textbf{Proof of Proposition \ref{prop:lowweight}.} 
Noting that modular forms of odd weight are necessarily cusp forms, 
we see that $M_1(\Gamma(1,6))=M_3(\Gamma (1,6))=\{ 0\}$ 
by Lemma \ref{lem:no_cusp_forms}. 
Also we see that $M_2(\Gamma (1,6))=\bC E_2$ by Lemma \ref{lem:no_cusp_forms} 
because if there is a non-zero element $f$ of $M_2(\Gamma(1,6))$ 
which is linearly independent of $E_2$, 
then we can assume that $f$ is a cusp form by adjusting it by $E_2$. 

Next, we prove $M_4(\Gamma (1,6))=\bC {E_2}^2 \oplus \bC E_4$.  
We assume that there is a non-zero element $f\in M_4(\Gamma(1,6))$ 
which is linearly independent of ${E_2}^2$ and $E_4$. 
Then we can assume that $C_f(0,0,0)=C_f(2,1,-1)=0$ 
by adjusting them by ${E_2}^2$ and $E_4$ (cf. the table on page \pageref{tab:weight2,4,6}).  
Then the Fourier coefficients of $f^2\in S_8(\Gamma(1,6))$ are 
\begin{align*} 
C_{f^2}(0,0,0) &= C_{f}(0,0,0)\cdot C_f(0,0,0) =0, \\ 
C_{f^2}(2,1,-1) &= 2\cdot C_{f}(0,0,0)\cdot C_f(2,1,-1)=0, \\ 
C_{f^2}(2,0,-1) &= 2\cdot C_{f}(0,0,0)\cdot C_f(2,0,-1)=0, \\ 
C_{f^2}(4,2,-2) &= 2\cdot C_{f}(0,0,0)\cdot C_f(4,2,-2)+C_f(2,1,-1)^2=0. 
\end{align*} 
Hence we have $f^2=0$, and therefore $f=0$. 
This contradicts the assumption. 
\begin{flushright} $\square$ \end{flushright}

%%%%%%%%%%%%%%%%%%%%%%%%%%%%%%%%%%%%%%%%%%%%%%%%%%%%%%%%%%%%%%%%%%%%%%%%%%%%
%%%%%%%%%%%%%%%%%%%%%%%%%%%%%%%%%%%%%%%%%%%%%%%%%%%%%%%%%%%%%%%%%%%%%%%%%%%%
%
\section{Proof of Proposition \ref{prop:weight5,10}} 
\label{sec:weight5,10}    
%
%%%%%%%%%%%%%%%%%%%%%%%%%%%%%%%%%%%%%%%%%%%%%%%%%%%%%%%%%%%%%%%%%%%%%%%%%%%%
%%%%%%%%%%%%%%%%%%%%%%%%%%%%%%%%%%%%%%%%%%%%%%%%%%%%%%%%%%%%%%%%%%%%%%%%%%%%

In this section, we will prove Propositin \ref{prop:weight5,10}, that is, 
we will determine the spaces of weight $5$ and $10$. 
By the dimension formulla, we have $\mathrm{dim}_{\bC}M_5(\Gamma(1,6))=2$ and 
$\mathrm{dim}_{\bC}M_{10}(\Gamma(1,6))=7$.  
We can obtain a $6$-dimensional subspace $V$ of $M_{10}(\Gamma(1,6))$ 
by products of Eisenstein series $E_k$'s: 
\begin{align*} 
V =\bC {E_2}^5 \oplus \bC {E_2}^3E_4 \oplus \bC {E_2}^2E_6  
  \oplus \bC E_2{E_4}^2 \oplus \bC E_4E_6 \oplus \bC E_{10}.  
\end{align*}  
We define $\varphi_2$, $\varphi_4$ and $\varphi_6$ as in section \ref{sec:lowweight} 
and define $\varphi_{10}$ by  
\begin{align*} 
\varphi _{10} &= 
\tfrac{31513745731}{416023384089600}\cdot (E_{10}-{\varphi_2}^5)
+\tfrac{52522796831}{2889051278400}\cdot {\varphi_2}^3\varphi _4 \\ 
 &{\hspace{5mm}} +\tfrac{21884309761}{481508546400}\cdot {\varphi_2}^2\varphi _6
-\tfrac{829232949}{1671904675}\cdot \varphi_2{\varphi _4}^2
+\tfrac{318067693}{1671904675}\cdot \varphi _4\varphi _6. 
\end{align*}  
Then we have Fourier coefficients of them as in the following table. 

\vspace{5mm} 
\begin{center} 
%\resizebox{0.8\textwidth}{!}{ 
\begin{longtable}{|c|c||c|c|c|c|c|c|} \hline 
$m_{\eta}$&$\eta$&${\varphi _2}^5$&${\varphi _2}^3\varphi _4$&${\varphi _2}^2\varphi _6$&$\varphi _2{\varphi _4}^2$&$\varphi _4\varphi _6$&$\varphi _{10}$ \\ \hline\hline 
0&[0, 0, 0]&1&0&0&0&0&0 \\ \hline
3&[2, 1, $-$1]&240&1&0&0&0&0 \\ \hline
4&[2, 0, $-$1]&360&$-$1&1&0&0&0 \\ \hline
12&[4, 2, $-$2]&24000&142&$-$6&1&0&0 \\ \hline
16&[4, 0, $-$2]&52920&$-$210&186&1&$-$1&0 \\ \hline
19&[4, 1, $-$2]&69840&67&80&$-$2&1&1 \\ \hline
24&[5, 1, $-$2]&151200&$-$132&228&4&$-$2&$-$4 \\ \hline
27&[6, 3, $-$3]&1291200&7236&$-$480&44&$-$6&$-$6 \\ \hline
36&[6, 0, $-$3]&4137480&$-$14373&11685&64&$-$36&$-$24 \\ \hline
40&[6, 2, $-$3]&5496480&12092&676&$-$28&$-$14&$-$12 \\ \hline
43&[6, 1, $-$3]&8018640&$-$5021&9008&$-$78&55&23 \\ \hline
48&[8, 4, $-$4]&40679520&155356&$-$6540&82&102&96 \\ \hline
51&[7, 2, $-$4]&19124640&10182&3648&$-$8&$-$44&20 \\ \hline
52&[7, 1, $-$4]&22154400&$-$17188&15652&112&4&$-$8 \\ \hline
64&[8, 0, $-$4]&189615960&$-$326884&275764&$-$654&118&320 \\ \hline 
\end{longtable} 
%} 
\end{center} 

\begin{lemma} \label{lem:weight5,10} 
For a non-zero element $f\in M_5(\Gamma(1,6))$, 
there is a non-zero element $\chi _f \in V$ such that 
$\chi _f$ is divisible by $f$ (i.e. the function $\chi _f/f$ is holomorphic). 
\end{lemma} 

\begin{proof} 
We can take some $g\in M_5(\Gamma(1,6))$ such that $M_5(\Gamma(1,6))=\bC f\oplus \bC g$. 
We have either $f^2\in V$ or $f^2\not\in V$. 
If $f^2\in V$, Lemma \ref{lem:weight5,10} holds for $\chi _f=f^2$. 
Hereafter we assume $f^2\not\in V$. 
Then we have $M_{10}(\Gamma(1,6))=V\oplus \bC f^2$. 
We have either $fg\in V$ or $fg \not\in V$. 
If $fg\in V$, then Lemma \ref{lem:weight5,10} holds for $\chi _f =fg$. 
If $fg\not\in V$, we can write $fg=x+r\cdot f^2$
for some $x\in V$ and some $r\in \bC ^{\times}$. 
Hence we have $V\ni x=fg-r\cdot f^2=f(g-r\cdot f)$ and $x\not= 0$. 
We see that Lemma \ref{lem:weight5,10} holds for $\chi_f=x$. 
\end{proof} 

\begin{lemma} \label{lem:weight5} 
We can find a basis $\chi_{5a}$, $\chi_{5b}$ of $M_5(\Gamma(1,6))$ which satisfy the following conditions: 
\begin{align*} 
C_{\chi_{5a}}(0,0,0) &= 0, & C_{\chi_{5a}}(2,1,-1) &= 0, & C_{\chi_{5a}}(2,0,-1) &= 1, \\ 
%C_{\chi_{5a}}(3,0,-2) &= 0, & C_{\chi_{5a}}(3,0,-1) &= 0, & C_{\chi_{5a}}(3,1,-2) &= -1, & C_{\chi_{5a}}(3,1,-1) &= -1, \\ 
C_{\chi_{5b}}(0,0,0) &= 0, & C_{\chi_{5b}}(2,1,-1) &= 1, & C_{\chi_{5b}}(2,0,-1) &= 0.  
%C_{\chi_{5b}}(3,0,-2) &= -1, & C_{\chi_{5b}}(3,0,-1) &= -1, & C_{\chi_{5b}}(3,1,-2) &= 0, & C_{\chi_{5b}}(3,1,-1) &= 0. 
\end{align*} 
\end{lemma} 

\begin{proof} 
Let $f,g$ be a basis of $M_5(\Gamma(1,6))$. 
We see from Lemma \ref{lem:weight5,10} that we can take $r$, $s\in\bC$ 
so that $f(rf+sg)\in V-\{ 0\}$. 
We put Fourier coefficients of them as  
\begin{align*} 
C_f(2,1,-1) &= \alpha , & C_f(2,0,-1) &= \beta , \\ 
C_g(2,1,-1) &= \gamma , & C_g(2,0,-1) &= \delta 
\end{align*} 
We assume $\alpha =\gamma =0$. 
Then Fourier coefficients of $h:=f(rf+sg)$ are as follows: 
\begin{align*} 
C_h(0,0,0) &= C_{f}(0,0,0)\cdot C_{f'}(0,0,0) = 0, \\ 
C_h(2,1,-1) &= C_{f}(0,0,0)\cdot C_{f'}(2,1,-1)+C_{f}(2,1,-1)\cdot C_{f'}(0,0,0) = 0, \\ 
C_h(2,0,-1) &= C_{f}(0,0,0)\cdot C_{f'}(2,0,-1)+C_{f}(2,0,-1)\cdot C_{f'}(0,0,0) = 0, \\ 
C_h(4,2,-2) &= C_{f}(0,0,0)\cdot C_{f'}(4,2,-2)+C_{f}(4,2,-2)\cdot C_{f'}(0,0,0) \\ 
 &\hspace{5mm} +C_{f}(2,1,-1)\cdot C_{f'}(2,1,-1)= 0, \\ 
C_h(4,0,-2) &= C_{f}(0,0,0)\cdot C_{f'}(4,0,-2)+C_{f}(4,0,-2)\cdot C_{f'}(0,0,0) \\ 
 &\hspace{5mm} +C_{f}(2,0,-1)\cdot C_{f'}(2,0,-1)= 0, \\ 
C_h(4,1,-2) &= C_{f}(0,0,0)\cdot C_{f'}(4,1,-2)+C_{f}(4,1,-2)\cdot C_{f'}(0,0,0) \\ 
 &\hspace{5mm} +C_{f}(2,0,-1)\cdot C_{f'}(2,1,-1)+C_{f}(2,1,-1)\cdot C_{f'}(2,0,-1)= 0, 
\end{align*} 
where $f':=rf+sg$.   
Hence we have $h=0$ because of the table of Fourier coefficients of the space of 
weight $10$. 
This contradicts the above. 
Hereafter we assume that either $\alpha$ or $\gamma$ is non-zero. 
We can assume that $\alpha=0$ and $\gamma=1$. 
If $\beta=0$, then the Fourier coefficients of $h$ satisfy the same condition as above.  
So we have $\beta\not=0$. 
We can assume $\beta =1$ and $\delta =0$. 
\end{proof} 

\noindent 
\textbf{Proof of Proposition \ref{prop:weight5,10}.} 
We take a basis $\chi_{5a}$ and $\chi_{5b}$ which satisfy the condition of 
Lemme \ref{lem:weight5}. 
Then we can verify that Fourier coefficients are as follows: 
{\small 
\begin{align*} 
C_{\chi_{5a}}(0,0,0) &= 0, & C_{\chi_{5a}}(2,1,-1) &= 0, & C_{\chi_{5a}}(2,0,-1) &= 1, \\ 
C_{\chi_{5a}}(3,0,-2) &= 0, & C_{\chi_{5a}}(3,0,-1) &= 0, & C_{\chi_{5a}}(3,1,-2) &= -1, & C_{\chi_{5a}}(3,1,-1) &= -1, \\ 
C_{\chi_{5b}}(0,0,0) &= 0, & C_{\chi_{5b}}(2,1,-1) &= 1, & C_{\chi_{5b}}(2,0,-1) &= 0. \\  
C_{\chi_{5b}}(3,0,-2) &= -1, & C_{\chi_{5b}}(3,0,-1) &= -1, & C_{\chi_{5b}}(3,1,-2) &= 0, & C_{\chi_{5b}}(3,1,-1) &= 0. 
\end{align*}}  
By Lemma \ref{lem:weight5,10}, 
we have $f:=\chi_{5a}(\alpha \chi_{5a}+\beta \chi_{5b})\in V$ for 
some $\alpha$, $\beta\in \bC$. 
Fourier coefficients of $f$ are 
\[ C_f(0,0,0)=C_f(2,1,-1)=C_f(2,0,-1)=C_f(4,2,-2)=0, \] 
\[ C_f(4,0,-2)=\alpha,\hspace{5pt} C_f(4,1,-2)=\beta \] 
by the same calculation as in the proof of Lemma \ref{lem:weight5}.  
We can see from the table of Fourier coefficients of the space of weight $10$ 
that 
$f= -\alpha\varphi_4\varphi_6 + (\alpha +\beta )\varphi_{10}$ and 
$C_f(5,1,-2)=-2\alpha -4\beta$. 
On the other hand, we have 
\begin{align*} 
C_f(5,1,-2) &= C_{\chi_{5a}}(0,0,0)\cdot C_{f'}(5,1,-2)+C_{\chi_{5a}}(5,1,-2)\cdot C_{f'}(0,0,0) \\ 
 &\hspace{5mm} +C_{\chi_{5a}}(2,0,-1)\cdot C_{f'}(3,1,-1)+C_{\chi_{5a}}(3,1,-1)\cdot C_{f'}(2,0,-1) \\ 
 &\hspace{5mm} +C_{\chi_{5a}}(2,1,-1)\cdot C_{f'}(3,0,-1)+C_{\chi_{5a}}(3,0,-1)\cdot C_{f'}(2,1,-1) \\ 
 &= -2\alpha ,
\end{align*}  
where $f'=\alpha\chi_{5a}+\beta\chi_{5b}$.  
Hence we have $\beta =0$, and therefore 
we can assume $f={\chi_{5a}}^2$ and 
\begin{align*} 
f &= \varphi _{10} -\varphi _4\varphi _6 \\ 
  &= \tfrac{31513745731}{416023384089600}E_{10} -\tfrac{126433528597}{311423218947072}{E_2}^5+\tfrac{11304517601}{14285468759040}{E_2}^3E_4 \\ 
&{\hspace{5mm}} -\tfrac{41742579637}{1557116094735360}{E_2}^2E_6-\tfrac{38947571}{120147846816}E_2{E_4}^2-\tfrac{1000259890201}{9083177219289600}E_4E_6.   
\end{align*} 
If $\chi_{5a}\chi_{5b}\in V$, then we have $\chi_{5a}(\chi_{5a}+\chi_{5b})\in V$ 
and this contradicts the above argument. 
Hence $\chi_{5a}\chi_{5b}\not\in V$ and $M_{10}(\Gamma(1,6))=V\oplus \bC \chi_{5a}\chi_{5b}$.  
We put ${\chi_{5b}}^2=v+r\chi_{5a}\chi_{5b}$ for some $v\in V$ and $r\in \bC$. 
Then $v=\chi_{5b}(\chi_{5b}-r\chi_{5a})$ and 
\[ C_f(0,0,0)=C_f(2,1,-1)=C_f(2,0,-1)=C_f(4,0,-2)=0, \] 
\[ C_f(4,2,-2)=1,\hspace{5pt} C_f(4,1,-2)=-r \] 
by the same calculation as above. 
Hence we have $v=\varphi_2{\varphi_4}^2+\varphi_4\varphi_6+(-r+1)\varphi_{10}$ and 
$C_v(5,1,-2)=-4r-2$. 
On the other hand,  
we have 
\begin{align*} 
C_v(5,1,-2) &= C_{\chi_{5b}}(0,0,0)\cdot C_{v'}(5,1,-2)+C_{\chi_{5b}}(5,1,-2)\cdot C_{v'}(0,0,0) \\ 
 &\hspace{5mm} +C_{\chi_{5b}}(2,0,-1)\cdot C_{v'}(3,1,-1)+C_{\chi_{5b}}(3,1,-1)\cdot C_{v'}(2,0,-1) \\ 
 &\hspace{5mm} +C_{\chi_{5b}}(2,1,-1)\cdot C_{v'}(3,0,-1)+C_{\chi_{5b}}(3,0,-1)\cdot C_{v'}(2,1,-1) \\ 
 &= -2 ,
\end{align*}  
where $v'=\chi_{5b}+r\chi_{5a}$.  
Hence we have $r=0$, and therefore $v={\chi_{5b}}^2$ and 
\begin{align*} 
{\chi_{5b}}^2 &= \varphi_2{\varphi_4}^2+\varphi_4\varphi_6+\varphi_{10} \\ 
&= \tfrac{31513745731}{416023384089600}E_{10} +\tfrac{266799861}{1281577032704}{E_2}^5-\tfrac{261925781}{1587274306560}{E_2}^3E_4 \\ 
 &{\hspace{5mm}} -\tfrac{1914649869}{6407885163520}{E_2}^2E_6+\tfrac{935053847}{51903869824512}E_2{E_4}^2+\tfrac{551346719209}{3406191457233600}E_4E_6.  
\end{align*}   
\begin{flushright} $\square$ \end{flushright}

%%%%%%%%%%%%%%%%%%%%%%%%%%%%%%%%%%%%%%%%%%%%%%%%%%%%%%%%%%%%%%%%
%%%%%%%%%%%%%%%%%%%%%%%%%%%%%%%%%%%%%%%%%%%%%%%%%%%%%%%%%%%%%%%%
%
\section{Proof of Proposition \ref{prop:weight15,20}} 
\label{sec:weight15,20} 
%
%%%%%%%%%%%%%%%%%%%%%%%%%%%%%%%%%%%%%%%%%%%%%%%%%%%%%%%%%%%%%%%%
%%%%%%%%%%%%%%%%%%%%%%%%%%%%%%%%%%%%%%%%%%%%%%%%%%%%%%%%%%%%%%%%

In this section, we will prove Proposition \ref{prop:weight15,20}, that is, 
we will determine the spaces of weight $15$ and $20$. 
By the result of the dimension formula,  
we have $\mathrm{dim}_{\bC}M_{20}(\Gamma(1,6))=28$. 
We can verify that the subspace $V$ of $M_{20}(\Gamma(1,6))$ spanned by 
all products of $E_2$, $E_4$, $\chi_{5a}$, $\chi_{5b}$ and $E_6$ is of dimension $26$. 
If we put $\delta _{20a}:=\{ E_2,E_4,\chi_{5a},E_6\}_*$ and 
$\delta _{20b}:=\{ E_2,E_4,\chi_{5b},E_6\}_*$, 
then we can verify that the complementary space of $V$ in $M_{20}(\Gamma(1,6))$ 
is spanned by $\delta_{20a}$ and $\delta_{20b}$ by calculating Fourier coefficients 
of them. 

By Proposition \ref{prop:weight5,10}, 
we see that $E_2$, $E_4$, $E_6$ and ${\chi_{5a}}^2-{\chi_{5b}}^2$ 
are algebraically dependent over $\bC$, 
so we have $\{ E_2,E_4,E_6,{\chi_{5a}}^2-{\chi_{5b}}^2\}_* =0$. 
By an elementary property of the differential calculus, 
we have 
\begin{align*} 
\{ E_2,E_4,{\chi_{5a}}^2,E_6\} &= 2\cdot \chi_{5a}\cdot \{ E_2,E_4,\chi_{5a},E_6\}_* \\ 
\rotatebox[origin=c]{90}{$=$} \hspace{13mm} & \\ 
\{ E_2,E_4,{\chi_{5b}}^2,E_6\} &= 2\cdot \chi_{5b}\cdot \{ E_2,E_4,\chi_{5b},E_6\}_* .
\end{align*} 
Hence we see that there is a cusp form $\chi_{15}$ such that 
$\{ E_2,E_4,\chi_{5a},E_6\}_* =\chi_{5b}\chi_{15}$ and 
$\{ E_2,E_4,\chi_{5b},E_6\}_* =\chi_{5a}\chi_{15}$. 

By the result of the dimension formula,  
we have $\mathrm{dim}_{\bC}M_{15}(\Gamma(1,6))=13$. 
We see that the subspace $U$ of $M_{15}(\Gamma(1,6))$ spanned by 
all products of $E_2$, $E_4$, $\chi_{5a}$, $\chi_{5b}$ and $E_6$ is of dimension $12$. 
If $\chi_{15}\in U$, then we see that $\delta_{20a}=\chi_{5b}\chi_{15}\in V$,  
but this is not the case. 
Hence we see that  $M_{15}(\Gamma(1,6))=U\oplus \bC \chi_{15}$. 
\begin{flushright} $\square$ \end{flushright}

%%%%%%%%%%%%%%%%%%%%%%%%%%%%%%%%%%%%%%%%%%%%%%%%%%%%%%%%%%%%%%%%
%%%%%%%%%%%%%%%%%%%%%%%%%%%%%%%%%%%%%%%%%%%%%%%%%%%%%%%%%%%%%%%%
%
\section{Proof of Theorem \ref{thm:main}} 
\label{sec:main} 
%
%%%%%%%%%%%%%%%%%%%%%%%%%%%%%%%%%%%%%%%%%%%%%%%%%%%%%%%%%%%%%%%%
%%%%%%%%%%%%%%%%%%%%%%%%%%%%%%%%%%%%%%%%%%%%%%%%%%%%%%%%%%%%%%%%

In this section, we will prove Theorem \ref{thm:main}. 
First, we calculate the generating function of $\mathrm{dim}_{\bC}M_k(\Gamma(1,6))$. 
From the dimension formula in subsection \ref{subsec:dimension_formula} and 
Proposition \ref{prop:lowweight}, 
we see that 
\begin{align*} 
\sum_{k=0}^{\infty} \mathrm{dim}_{\bC}M_k(\Gamma (1,6))t^k  
&= 1+t^2+2t^4+\sum_{k=5}^{\infty} \mathrm{dim}_{\bC}S_k(\Gamma (1,6))t^k 
+\sum_{k=3}^{\infty} t^{2k} \\ 
&= \frac{(1+t^5)(1+t^{15})}{(1-t^2)(1-t^4)(1-t^5)(1-t^6)}.  
\end{align*} 
By the results of the previous sections, we have obtained  
\begin{align} \label{eq:sec:main_1} 
\bigoplus _{k=0}^{\infty} M_k(\Gamma (1,6)) \supseteq 
\bC [E_2,E_4,\chi_{5a},\chi_{5b},E_6,\chi_{15}]. 
\end{align} 
We do not mean that six modular forms in the right side of (\ref{eq:sec:main_1}) 
are algebraically independent over $\bC$. 
We need to determine the precise structure of the right side of (\ref{eq:sec:main_1}).  
\begin{lemma} \label{lem:proofmain} \ \\ 
$(i)$\ $E_2$, $E_4$, $\chi_{5a}$ and $E_6$ are algebraically independent over $\bC$. \\ 
$(ii)$\ ${\chi_{5b}}^2$, ${\chi_{15}}^2\in \bC [E_2,E_4,\chi_{5a},E_6]$. \\ 
$(iii)$\ $1$ and $\chi_{5b}$ are linearly independent over $\bC[E_2,E_4,\chi_{5a},E_6]$. \\ 
$(iv)$\ $1$ and $\chi_{15}$ are linearly independent over $\bC [E_2,E_4,\chi_{5a},\chi_{5b},E_6]$. 
\end{lemma} 

\begin{proof} 
(i)\ This is followed from Proposition \ref{prop:diffoperator} 
because $\{ E_2,E_4,\chi_{5a},E_6\}_*$ $=$ $\chi_{5b}\chi_{15}$ $\not= 0$. \\ 
(ii)\ This is proved by comparison of Fourier coefficients. In fact, 
we give the expression of ${\chi_{5a}}^2$ and ${\chi_{5b}}^2$ 
by $E_2$, $E_4$, $\chi_{5a}$ and $E_6$ in Appendix. \\ 
(iii)\ If $\alpha +\beta \chi_{5b}=0$ for $\alpha, \beta \in \bC [E_2,E_4,\chi_{5a},E_6]$, 
then we have $\alpha ^2=\beta ^2{\chi_{5b}}^2$. 
We see from (i) that $\alpha^2$ and $\beta^2$ can be regarded as the squares of 
polynonials with four variables $E_2$, $E_4$, $\chi_{5a}$ and $E_6$, 
while ${\chi_{5b}}^2$ is not so. 
Hence we have $\alpha =\beta =0$. \\ 
(iv)\ If $f+\chi_{5b}g=\chi_{15}(h+\chi_{5b}j)$ for 
$f,g,h,j\in \bC [E_2,E_4,\chi_{5a},E_6]$, then we have 
\[ 2(fg-{\chi_{15}}^2hj)\chi_{5b}=-f^2-{\chi_{5b}}^2g^2+{\chi_{15}}^2h^2+{\chi_{5b}}^2{\chi_{15}}^2j^2. \] 
We see from (ii) and (iii) that 
\begin{align} 
&fg ={\chi_{15}}^2hj, \label{eq:sec:main_2} \\ 
&f^2+{\chi_{5b}}^2g^2={\chi_{15}}^2(h^2+{\chi_{5b}}^2j^2). \label{eq:sec:main_3} 
\end{align} 
We can see that ${\chi_{15}}^2$ is irreducible as a polynomial with $4$ variables 
$E_2$, $E_4$, $\chi_{5a}$ and $E_6$. 
We see from (\ref{eq:sec:main_2}) that either $f$ or $g$ is divisible by ${\chi_{15}}^2$.  
We see from (\ref{eq:sec:main_3}) that both $f$ and $g$ are divisible by 
${\chi_{15}}^2$. 
By dividing (\ref{eq:sec:main_2}) and (\ref{eq:sec:main_3}) by ${\chi_{15}}^2$, 
we obtain equations of the same shape as (\ref{eq:sec:main_2}) and (\ref{eq:sec:main_3}). 
We can repeat this infinitely, so $f$, $g$, $h$ and $j$ must be $0$.  
\end{proof} 
We see from Lemma \ref{lem:proofmain} that 
\begin{align*} 
\bC [E_2,E_4,\chi_{5a},\chi_{5b},E_6,\chi_{15}] 
&= \bC [E_2,E_4,\chi_{5a},\chi_{5b},E_6]\oplus \chi_{15} \bC [E_2,E_4,\chi_{5a},\chi_{5b},E_6] \\ 
&= \bC [E_2,E_4,\chi_{5a},E_6]\oplus \chi_{5b} \bC [E_2,E_4,\chi_{5a},E_6] \\ 
&\hspace{5mm} \oplus \bC [E_2,E_4,\chi_{5a},E_6]\oplus \chi_{5b} \bC [E_2,E_4,\chi_{5a},E_6]. 
\end{align*} 
Hence the generating function of $\mathrm{dim}_{\bC}M_k(\Gamma(1,6))$ is the same 
as that of dimensions of right side of (\ref{eq:sec:main_1}). 
We have completed the proof of Theorem \ref{thm:main}.

%%%%%%%%%%%%%%%%%%%%%%%%%%%%%%%%%%%%%%%%%%%%%%%%%%%%%%%%%%%%%%%%%
%%%%%%%%%%%%%%%%%%%%%%%%%%%%%%%%%%%%%%%%%%%%%%%%%%%%%%%%%%%%%%%%%
%
\section{Appendix} \label{sec:appendix} 
%
%%%%%%%%%%%%%%%%%%%%%%%%%%%%%%%%%%%%%%%%%%%%%%%%%%%%%%%%%%%%%%%%%
%%%%%%%%%%%%%%%%%%%%%%%%%%%%%%%%%%%%%%%%%%%%%%%%%%%%%%%%%%%%%%%%%

We give a table of Fourier coefficients of the generators of 
$\bigoplus _{k=0}^{\infty} M_{k}(\Gamma(1,6))$ in Theorem \ref{thm:main}. 

\begin{center} 
\begin{longtable}{|c||c|c|c|c|c|c|} \hline 
 $\eta$ & $E_2$&$E_4$&$E_6$&$\chi_{5a}$&$\chi_{5b}$&$\chi_{15}$ \\ \hline \hline 
(0,0,0)&1&1&1&0&0&0\\ \hline
(2,1,-1)&48&960/13&2016/341&0&1&0\\ \hline
(2,0,-1)&72&2160/13&7560/341&1&0&0\\ \hline
(4,2,-2)&192&35520/13&1066464/341&0&16&0\\ \hline
(4,0,-2)&216&71280/13&3878280/341&6&0&0\\ \hline
(4,1,-2)&144&95040/13&8134560/341&$-16$&$-27$&0\\ \hline
(5,1,-2)&288&198720/13&24101280/341&0&0&1\\ \hline
(6,3,-3)&192&234240/13&39682944/341&0&12&0\\ \hline
(6,0,-3)&360&546480/13&149423400/341&81&0&0\\ \hline
(6,2,-3)&288&682560/13&239023008/341&40&0&0\\ \hline
(6,1,-3)&144&717120/13&10348128/11&16&135&0\\ \hline
(8,4,-4)&480&1141440/13&546063840/341&0&256&0\\ \hline
(7,2,-4)&288&1157760/13&694612800/341&0&54&112\\ \hline
(7,1,-4)&288&1304640/13&778117536/341&$-68$&0&162\\ \hline
(8,0,-4)&504&2283120/13&1985686920/341&$-92$&0&0\\ \hline
(8,3,-4)&144&2168640/13&2360177568/341&128&$-189$&0\\ \hline
(8,1,-4)&336&3024960/13&3938762016/341&0&85&0\\ \hline
(8,2,-4)&576&3516480/13&4303182240/341&$-224$&$-432$&0\\ \hline
(9,3,-5)&576&4544640/13&6765837120/341&0&0&$-3564$\\ \hline
(9,1,-5)&288&371520&8301345696/341&$-112$&0&$-5103$\\ \hline
(9,2,-5)&288&4752000/13&9366960960/341&112&162&$-1296$\\ \hline
(10,2,-6)&864&6557760/13&12363956640/341&0&0&14976\\ \hline
(10,0,-5)&720&6968160/13&14784532560/341&890&0&0\\ \hline
(12,6,-6)&768&8666880/13&20992277376/341&0&192&0\\ \hline
%(10,1,-5)&288&8484480/13&26849648448/341&$-320$&$-810$&0\\ \hline
%(10,2,-5)&576&10972800/13&33668018496/341&0&0&0\\ \hline
%(11,4,-6)&288&10281600/13&36472838976/341&0&$-702$&33824\\ \hline
%(11,1,-6)&576&13875840/13&51688348992/341&0&0&6088\\ \hline
%(11,3,-6)&576&14636160/13&58915080000/341&656&0&$-54594$\\ \hline
%(11,2,-6)&432&13763520/13&63024726240/341&192&837&$-29808$\\ \hline 
\end{longtable} 
\end{center} 

In Lemma \ref{lem:proofmain}, we state that ${\chi_{5b}}^2$ and ${\chi_{15}}^2$ 
can be written as polynomials of $4$ variables $E_2$, $E_4$, $\chi_{5a}$ and 
$E_6$. In fact, we have the following relations. 
These are followed from the comparison of Fourier coefficients. 

\vspace{15pt} 
\noindent 
${\chi_{5b}}^2=(5005/8149248)*{E_2}^5-(15587/16298496)*{E_2}^3E_4-(4433/16298496)*\\{E_2}^2E_6+(1859/5432832)*E_2{E_4}^2+(4433/16298496)*E_4E_6+{\chi_{5a}}^2$

\vspace{10pt} 
\begin{longtable}{cl}  
${\chi_{15}}^2$ &$=  
(7193626131746618585/222607917767232721152)*{E_2}^{15}$ \\ 
&$-(307986483294442487/1426973831841235392)*{E_2}^{13}{E_4}$ \\ 
&$+(1416328854305111/54400761917701056)*{E_2}^{12}{E_6}$ \\ 
&$+(4087366592607641/6860451114621324)*{E_2}^{11}{E_4}^2$ \\ 
&$-(192607575137275/1394891331223104)*{E_2}^{10}{E_4}{E_6}$ \\ 
&$+(50704311727294/69507316593)*{E_2}^{10}{\chi_{5a}}^2$ \\ 
&$-(52003816542174887/59873027909422464)*{E_2}^9{E_4}^3$ \\ 
&$+(2912260461769/319066052303232)*{E_2}^9{E_6}^2$ \\ 
&$+(1922370985523/6706208323188)*{E_2}^8{E_4}^2{E_6}$ \\ 
&$-(20825649443174/5346716661)*{E_2}^8{E_4}{\chi_{5a}}^2$ \\ 
&$+(102989732952024139/146356290445254912)*{E_2}^7{E_4}^4$ \\ 
&$-(96923094941/2727060276096)*{E_2}^7{E_4}{E_6}^2$ \\ 
&$+(27583081580/203833773)*{E_2}^7{E_6}{\chi_{5a}}^2$ \\ 
&$-(92968372638167/321897999513024)*{E_2}^6{E_4}^3{E_6}$ \\ 
&$+(65651791909/36815313727296)*{E_2}^6{E_6}^3$ \\ 
&$+(3387092572918/411285897)*{E_2}^6{E_4}^2{\chi_{5a}}^2$ \\ 
&$-(7304217732454747/24392715074209152)*{E_2}^5{E_4}^5$ \\ 
&$+(30622846693/629321602176)*{E_2}^5{E_4}^2{E_6}^2$ \\ 
&$-(256204744/505791)*{E_2}^5{E_4}{E_6}{\chi_{5a}}^2$ \\ 
&$-(10936889634816/19651489)*{E_2}^5{\chi_{5a}}^4$ \\ 
&$+(14944942065833/107299333171008)*{E_2}^4{E_4}^4{E_6}$ \\ 
&$-(27494911499/6135885621216)*{E_2}^4{E_4}{E_6}^3$ \\ 
&$-(1176607216174/137095299)*{E_2}^4{E_4}^3{\chi_{5a}}^2$ \\ 
&$+(10349644/597753)*{E_2}^4{E_6}^2{\chi_{5a}}^2$ \\ 
&$+(36987323269/710702030016)*{E_2}^3{E_4}^6$ \\ 
&$-(49717185583/1887964806528)*{E_2}^3{E_4}^3{E_6}^2$ \\ 
&$+(1709446981/8862945897312)*{E_2}^3{E_6}^4$ \\ 
&$+(773604236/1206117)*{E_2}^3{E_4}^2{E_6}{\chi_{5a}}^2$ \\ 
&$+(2503569715200/1511653)*{E_2}^3{E_4}{\chi_{5a}}^4$ \\ 
&$-(26102557/1042085088)*{E_2}^2{E_4}^5{E_6}$ \\ 
&$+(2820958987/943982403264)*{E_2}^2{E_4}^2{E_6}^3$ \\ 
&$+(509138188/116281)*{E_2}^2{E_4}^4{\chi_{5a}}^2$ \\ 
&$-(2420960/45981)*{E_2}^2{E_4}{E_6}^2{\chi_{5a}}^2$ \\ 
&$-(31993344000/57629)*{E_2}^2{E_6}{\chi_{5a}}^4$ \\ 
&$+(18421/4583952)*{E_2}{E_4}^4{E_6}^2$ \\ 
&$-(159653813/681765069024)*{E_2}{E_4}{E_6}^4$ \\ 
&$-(843440/3069)*{E_2}{E_4}^3{E_6}{\chi_{5a}}^2$ \\ 
&$-(136400/66417)*{E_2}{E_6}^3{\chi_{5a}}^2$ \\ 
&$-(137631744000/116281)*{E_2}{E_4}^2{\chi_{5a}}^4$ \\ 
&$-(4433/20627784)*{E_4}^3{E_6}^3$ \\ 
&$+(39651821/4431472948656)*{E_6}^5$ \\ 
&$-(301621736/348843)*{E_4}^5{\chi_{5a}}^2$ \\ 
&$+(1100/27)*{E_4}^2{E_6}^2{\chi_{5a}}^2$ \\ 
&$+(3018240000/4433)*{E_4}{E_6}{\chi_{5a}}^4$ \\ 
&$+(40993977139200000/19651489)*{\chi_{5a}}^6$  
\end{longtable}

%%%%%%%%%%%%%%%%%%%%%%%%%%%%%%%%%%%%%%%%%%%%%%%%%%%%%%%%%%%%%%%
%%%%%%%%%%%%%%%%%%%%%%%%%%%%%%%%%%%%%%%%%%%%%%%%%%%%%%%%%%%%%%%
%

\vspace*{5mm}
\noindent
Hidetaka Kitayama\\
Department of Mathematics\\
Osaka University\\ 
Machikaneyama 1-1, Toyonaka\\ 
Osaka, 560-0043, Japan\\
E-mail: \texttt{h-kitayama@cr.math.sci.osaka-u.ac.jp}\\

\end{document}